\documentclass[12pt]{article}

\usepackage{epsfig}
\usepackage{latexsym}
\usepackage{caption}
\usepackage{amsmath}
\usepackage{amsfonts}
\usepackage{amssymb}
\usepackage{graphicx}

\usepackage{color}

\makeatletter

\newcommand{\Rmnum}[1]{\expandafter\@slowromancap\romannumeral #1@}
\makeatother

\begin{document}

\newtheorem{theorem}{Theorem}[section]
\newtheorem{observation}[theorem]{Observation}
\newtheorem{corollary}[theorem]{Corollary}
\newtheorem{algorithm}[theorem]{Algorithm}
\newtheorem{definition}[theorem]{Definition}
\newtheorem{guess}[theorem]{Conjecture}
\newtheorem{claim}[theorem]{Claim}
\newtheorem{problem}[theorem]{Problem}
\newtheorem{question}[theorem]{Question}
\newtheorem{lemma}[theorem]{Lemma}
\newtheorem{proposition}[theorem]{Proposition}
\newtheorem{fact}[theorem]{Fact}

\makeatletter
  \newcommand\figcaption{\def\@captype{figure}\caption}
  \newcommand\tabcaption{\def\@captype{table}\caption}
\makeatother

\newtheorem{acknowledgement}[theorem]{Acknowledgement}

\newtheorem{axiom}[theorem]{Axiom}
\newtheorem{case}[theorem]{Case}
\newtheorem{conclusion}[theorem]{Conclusion}
\newtheorem{condition}[theorem]{Condition}
\newtheorem{conjecture}[theorem]{Conjecture}
\newtheorem{criterion}[theorem]{Criterion}
\newtheorem{example}[theorem]{Example}
\newtheorem{exercise}[theorem]{Exercise}
\newtheorem{notation}[theorem]{Notation}
\newtheorem{solution}[theorem]{Solution}
\newtheorem{summary}[theorem]{Summary}

\newenvironment{proof}{\noindent {\bf
Proof.}}{\rule{3mm}{3mm}\par\medskip}
\newcommand{\remark}{\medskip\par\noindent {\bf Remark.~~}}
\newcommand{\pp}{{\it p.}}
\newcommand{\de}{\em}
\newcommand{\mad}{\rm mad}



\title{Anti-magic labeling of  regular graphs}

\author{
    Feihuang Chang\thanks{Grant number: NSC 102-2115-M-003-008-}\\
   Division of Preparatory Programs
   for  Overseas Chinese Students\\
    Taiwan Normal University\\
    {\tt feihuang0228@gmail.com}
    \and
   Yuchang Liang  \\
   Department of  Applied Mathematics\\
   National Sun Yat-sen University\\
   Taiwan\
   {\tt chase2369219@hotmail.com}
\and
   Zhishi Pan\thanks{Grant number:NSC 102-2115-M-032 -006 -} \\
   Department of mathematics\\
   Tamkang University\\
    Taiwan\\
   {\tt zhishi.pan@gmail.com}
\and
   Xuding Zhu
   \thanks{Grant number: NSF11171310.} \\
   Department of   Mathematics\\
   Zhejiang Normal University\\
    Jinhua, China\\
   {\tt xudingzhu@gmail.com}
}

\maketitle

\begin{abstract}
A graph $G=(V,E)$ is antimagic if there is a one-to-one correspondence $f: E \to \{1,2,\ldots, |E|\}$ such that
for any two vertices $u,v$, $\sum_{e \in E(u)}f(e) \ne \sum_{e\in E(v)}f(e)$.
It is known that bipartite regular graphs are antimagic and non-bipartite regular graphs of odd degree at least three are antimagic.
Whether all non-bipartite regular graphs of even degree are antimagic remained an open problem. In
this paper, we solve this problem and prove that all even degree regular graphs  are antimagic.
\end{abstract}

Keywords: Anti-magic, regular graph,  labeling.

\section{Introduction}

 Suppose  $G=(V,E)$ is a graph and   $f: E \to \{1,2,\ldots,
 |E|\}$ is a bijective mapping. For each vertex $u$ of $G$, the {\em vertex-sum} $\varphi_f(u)$ at $u$ is
 defined as  $\varphi_f(u)=\sum_{e
 \in E(u)}f(e)$, where $E(u)$ is the set
 of edges incident to $u$.

 If $\varphi_f(u) \ne \varphi_f(v)$ for any two
 distinct vertices $u,v$ of $G$, then $f$ is called   an {\em antimagic labeling} of $G$.
 A graph $G$ is called {\em antimagic} if $G$ has an antimagic
 labeling. The problem of antimagic labeling of graphs was introduced by Hartsfield and Ringel \cite{HR1990}.
 They put forth two
 conjectures concerning antimagic labeling of graphs.

\begin{guess} [\cite{HR1990}]
\label{g1} Every   connected graph other than $K_2$ is
 antimagic.
 \end{guess}

\begin{guess} [\cite{HR1990}]
\label{g2} Every  tree other than $K_2$ is
 antimagic.
 \end{guess}

The conjectures have received much attention, but both conjectures
remain open. Some special cases of the conjectures are verified \cite{AKLRY2004,survey,KLR2009,Dan2005,LWZ2012,LZ2013}.
Recently, it is shown by Eccles \cite{Ecc2014} that graphs of large linear size are antimagic.
As vertices of the same degree   seem more difficult to receive distinct vertex-sums,
a natural special case of Conjecture \ref{g1} is whether for $k \ge 2$,
every $k$-regular graph is antimagic. Cranston \cite{Cra2009} proved that for $k \ge 2$, every $k$-regular bipartite graph
is antimagic. For non-bipartite regular graphs,
Liang and Zhu \cite{LZ2014} proved that every cubic graph is antimagic, and the result was generalized by Cranston, Liang and Zhu
\cite{CLZ}, where it is proved that odd degree regular graphs are antimagic.
It remains an open problem whether every even degree regular non-bipartite graph is antimagic.
Hartsfield and Ringel \cite{HR1990}
proved that   every $2$-regular graph is
antimagic. This paper solves this problem and
proves that indeed every even degree regular graph is antimagic.

\section{A lemma}
In this section, we prove a   lemma  which will be used in the  proof of the main result. Assume
$G=(X \cup Y, E)$ is a bipartite graph. A {\em $Y$-link} is a path $p=(y,x,y')$ of length $2$ with
$y, y' \in Y$. A {\em  $Y$-link family} is a family $F$ of vertex disjoint $Y$-links. Vertices in
$F$ are said to be {\em covered by $F$}. For simplicity, we shall write a $Y$-link $p=(y,x,y')$ as
$p=yxy'$.

Assume each vertex of $X$ has degree at most $d$.
If $F$ is a $Y$-link family and $M$ is a matching in $G$ and $F \cup M$ covers every   $d$-degree vertex of $X$,
then we call $(F,M)$ a {\em $d$-covering pair} (or a {\em covering pair}, when $d$ is clear from the context).
A {\em  minimum $d$-covering pair} is a $d$-covering pair $(F,M)$ with
   minimum $|F|$.

Note that if $(F,M)$ is a minimum covering pair of $G$, then each  $d$-degree vertex of $X$ is covered by exactly one of $M $ or $F $,
and
if $yxy' \in F$, then both $y,y'$ are incident to $M$-edges (for otherwise we may replace $yxy'$ by a matching edge $xy$ or $xy'$).
Thus each component of the union graph $M  \cup F $ is either a single edge (a matching edge) or a $W$-shape graph (a $Y$-link together with two matching edges).

\begin{lemma}
\label{key} Assume $G=(X \cup Y, E)$ is a bipartite graph,  $d \ge 3$, each vertex in $X$ has degree  at most $d$ and each vertex in $Y$ has degree at most $d+1$.
Then $G$ has a $d$-covering pair.
 \end{lemma}

 To prove Lemma \ref{key}, we may assume each vertex of $Y$ has degree $d+1$
 and each vertex of $X$ has degree $d$, for otherwise, let $G' =(X' \cup Y', E')$ be a bipartite graph which contains $G$ as an induced subgraph
 and each vertex of $X'$ has degree $d$ and each vertex of $Y'$ has degree $d+1$   (such a graph $G'$ is easily seen to exist).
 Let $(F',M')$ be a $d$-covering pair of $G'$.
Let  $M = M' \cap E(G)$ and   $F=\{p \in F': p \subset G\}$. Then $(F,M)$ is a $d$-covering pair   of $G$.

For a $Y$-link  family $F$ in $G$, let
\begin{eqnarray*}
X_F =  V(F) \cap X, \ \
Y_F = N_G(X_F).
\end{eqnarray*}
For $1 \le i \le d+1$, let
\begin{eqnarray*}
Y_{F,i} = \{y \in Y_F: |N_G(y) \cap X_F| = i\},\ \
Y_{F, i^+} = \cup_{j=i}^{d+1}Y_{F,i}.
\end{eqnarray*}
Let
$$X'_F = \{x \in X -X_F: N_G(x) \cap Y_{F, 2^+} \ne \emptyset\}.$$

Note that $Y_F$ is the set of vertices of $Y$ that is adjacent to some vertices in $X_F$ (in $G$), and hence $Y_F$ may contain vertices
not covered by $F$.

 We choose $F$ so that
 \begin{itemize}
 \item[(1)]  $|Y_F|$ is maximum.
 \item[(2)]   Subject to (1), $|X'_F|$ is maximum.
 \end{itemize}

 \begin{lemma}
 \label{lem1}
 For any $y \in Y -Y_F$, $N_G(y) \subseteq X'_F$.
 \end{lemma}
 \begin{proof}
 Assume to the contrary that $y \in Y -Y_F$ and $x \in N_G(y)-X'_F$.

 A {\em $Y$-link sequence } is a sequence of $Y$-links $\{y_jx_jy'_j \in F: 0 \le j \le i-1 \}$
 such that $y_{j+1} \sim x_j$ for $0 \le j \le i-1$.

 Let  $y_0 \in N_G(x) - \{y\}$.
 Then $y_0$ must be an end vertex of a $Y$-link, say  $y_0x_0y'_0  \in F$, for otherwise,  $F' =F \cup
 \{yxy_0\}$ is a $Y$-link family   with $|Y_{F'}| > |Y_F|$ (as $y \in Y_{F'}-Y_F$), contrary to the choice of $F$.

 Next we show that if there is a $Y$-link sequence  $\{y_jx_jy'_j \in F: 0 \le j \le i-1 \}$ of length $i$ with $y_0 \in N_G(x)$,
 then it can be extended to  a $Y$-link sequence of length $i+1$.
 This implies that  there is an infinite $Y$-link sequence, which is a contradiction, as $G$ is finite.

 Assume $i \ge 1$ and $\{y_jx_jy'_j \in F: 0 \le j \le i-1 \}$ is a $Y$-link sequence, with $y_0 \in N_G(x)$.
 Let $y_{i} \in N_G(x_{i-1}) - \{y_{i-1}, y'_{i-1}\}$ (the vertex $y_i$ exists as $x_{i-1}$ has degree $d \ge 3$).

\begin{claim}
\label{cl0} For $0 \le j \le i-1$,  $y'_j \in Y_{F,1} \ \mbox{\rm and} \ y_{j+1} \in Y_{F,2} $.
\end{claim}
\begin{proof}
 First we show that  $y_i$ is  the end of a $Y$-link in $F$. Otherwise
 $$F'=F-\{y_jx_jy'_j: 0 \le j \le i-1\} \cup \{y_{j+1}x_jy'_j: 0 \le j \le i-1\} \cup \{y_0xy\}$$
 is a $Y$-link family with $|Y_{F'}| > |Y_F|$, contrary to the choice of $F$.

Now we prove that for   $0 \le j \le i-1$,  $y'_j \in Y_{F,1}$.

 Assume to the contrary that for some $0 \le j  \le i-1$,  $y'_j \notin Y_{F,1} $.  Then
  $y'_j \in Y_{F, 2^+}$. Let  $$F' = (F - \{y_lx_ly'_l: 0 \le l \le j\}) \cup \{y'_lx_ly_{l+1}: 0 \le l \le j-1\} \cup \{yxy_0\}.$$
  Then $F'$ is a $Y$-link family with
$Y_{F'}= Y_F \cup \{y\}$, contrary to the choice of $F$.

Next we prove that for $0 \le j \le i-1$,  $  y_{j+1} \in Y_{F,2} $.
First observe that  $y_{j+1}$ is adjacent to at least two vertices of $X_F$, namely $x_j$ and $x_{j+1}$.
So $y_{j+1} \in Y_{F, 2^+}$.

Assume to the contrary that there is an index $0 \le j \le i-1$ such that $y_{j+1} \in Y_{F,3^+}$.
Let
$$F' = (F - \{y_lx_ly'_l: 0 \le l \le j\}) \cup \{y'_lx_ly_{l+1}: 0 \le l \le j-1\} \cup \{yxy_0\}.$$
It is easy to verify that $Y_{F'} = (Y_F - \{y'_j\}) \cup \{y\}$, and $ Y_{F,2^+} =
 Y_{F',2^+}$.
 This implies that  $X'_{F'} = C_{F} \cup \{x_j\}$. So $|Y_{F'}| =
|Y_F|$ and $|X'_{F'}| > |X'_F|$,  contrary to the choice of $F$. This completes the proof of Claim
\ref{cl0}.
\end{proof}

By Claim \ref{cl0},   $y_i \notin \{ y_j, y'_j: 0 \le j \le i-1\}$. We have observed that $y_i$ is the end of a $Y$-link. Let $y_ix_iy'_i$ be the link
containing $y_i$. Then $\{y_jx_jy'_j: 0 \le j \le i\}$ is a  $Y$-link sequence.
This completes the proof of Lemma \ref{lem1}.
 \end{proof}

 \noindent
 {\bf Proof of Lemma \ref{key}}
 Let $F$ be a family of $Y$-links such that for any   $y \in Y -Y_F$, $N_G(y) \subseteq X'_F$.
 Let $G'=(X' \cap Y', E')$ be the subgraph of $G$ obtained  from $G$  by deleting all the vertices in $X$ covered by $F$.
 It suffices to show that $G'$ has a matching $M$ that covers every vertex of $X'$.
 By Hall's Theorem, it suffices to show that for any subset $Z$ of $X'$, $$|N_{G'}(Z)| \ge |Z|.$$
 Let $Z$ be a subset of $X'$ and $W=N_G(Z)$. Let $H$ be the subgraph of $G'$ induced by $Z  \cup W$.
 For $i=1,2,\ldots, d+1$, let $W_i = \{y \in W: d_H(y)=i\}$.

By the maximality of $Y_F$, $N_G(y) \cap N_G(y') = \emptyset$ for distinct $y,y' \in W_{d+1}$. Therefore  $$ |N_H(W_{d+1})| = (d+1) |W_{d+1}|. \eqno(1)$$
By Lemma \ref{lem1},  each vertex in $N_H(W_{d+1})$ is adjacent to a vertex in $W_1 \cup W_2 \cup \ldots \cup W_{(d-1)}$.
As each vertex in $W_1 \cup W_2\cup \ldots \cup W_{(d-1)}$ is adjacent to at most $(d-1)$ vertices in
 $N_H(W_{d+1})$, we have $$(d-1)|W_1 \cup W_2\cup \ldots \cup W_{(d-1)}| \ge |N_H(W_{d+1})|. \eqno(2) $$ (1) and (2) imply that  $$|W_1 \cup W_2\cup \ldots \cup W_{(d-1)}| \ge |W_{d+1}|. \eqno(3)$$
 As each vertex of $Z$ has degree $d$ in $H$, by using (3), we conclude that
 $$d|Z| = |E(H)| = \sum_{y \in W}d_H(y)=\sum_{i=1}^{d+1} i|W_i| \le d|W| + |W_{d+1}| - \sum_{i=1}^{d-1}|W_i| \le d|W|.$$
 Hence $|W|= |N_{G'}(Z)| \ge |Z|$, and $G'$ has a matching $M$ that covers every vertex of $X'$.
 This completes the proof of Lemma \ref{key}.
 \section{Proof of the main result}

This section proves the main result of this paper:

\begin{theorem}
\label{even}
For any  integer $k \ge 1$, all $(2k+2)$-regular graphs are antimagic.
\end{theorem}

The general idea of the proof of Theorem \ref{even} is the same as the proof for odd degree regular graphs
(cf. \cite{CLZ, LZ2014}).

Let $G$ be a $(2k+2)$-regular graph. We may assume   $G$ is connected.
We choose a vertex $v^*$ of $G$ and partition $V(G)$ into $L_0 \cup L_1 \cup \ldots \cup L_p$,
where $L_i=\{v \in V(G): d_G(v^*,v)=i\}$. We say an edge $e$ has distance $i$ to $v^*$ if $e$ connects two vertices in $L_i$
or one vertex in $L_i$ with one vertex in $L_{i-1}$.
We label the edges in order of decreasing distance to $v^*$, i.e.,   smaller labels are assigned to edges
further away from $v^*$. For edges of the same distance to $v^*$, the labels need to be carefully assigned.

For the labeling to be antimagic, it suffices to guarantee that for $i \ge 1$,
vertices in $L_i$ have smaller vertex sum than vertices in $L_{i-1}$, and distinct
vertices in $L_i$ have distinct vertex sums.

To make sure vertices in $L_i$ have distinct vertex sums, for each vertex $u \in L_i$,
we choose one edge $\sigma(u)$ connecting $u$ to a vertex in $L_{i-1}$.  We first label
other edges whose distance  to $v^*$ is at least $i$. Then  we label edges in $\{\sigma(u): u \in L_i\}$.
Right before we label these edges,
for each $u \in L_i$, $\sigma(u)$ is the only unlabeled edge incident to $u$. Order the vertices
of $L_i$ as $u_{i,1}, u_{i,2}, \ldots, u_{i, n_i}$ according to their partial vertex-sums, i.e., the partial vertex-sum of
$u_{i,j}$  (just before the labeling of edges $\sigma(u)$ for $u \in L_i$) does not exceed the partial vertex-sum of $u_{i, j+1}$.
Then we label the edges in $\{\sigma(u): u \in L_i\}$ in increasing order, i,e., the label of $\sigma(u_{i,j})$ is less than the label of
$\sigma(u_{i,j+1})$ for $=1,2,\ldots, p-1$. This ensures that vertices in $L_i$ have distinct vertex-sums.

The main difficulties are to choose the labels carefully to ensure that the label sums of vertices in $L_i$ are strictly less than the
vertex sums of vertices in $L_{i-1}$.  For this purpose, we use Lemma \ref{key}   to choose, among edges of the same distance to $v^*$,
some special edges that will receive
relatively large labels.

Below is the detailed argument.

For $i \ge 2$, $G[L_{i-1}, L_i]$ is the   bipartite graph consisting edges between $L_i$ and $L_{i-1}$.
In this bipartite graph, each vertex in $L_{i-1}$ has degree at most $2k+1$, and every vertex in $L_i$ has degree at most $2k+2$.

By Lemma \ref{key}, $G[L_{i-1}, L_i]$ has a $(2k+1)$-covering pair.
Let $(F_i,M_i)$ be a minimum $(2k+1)$-covering pair of $G[L_{i-1}, L_i]$.

Edges in $M_i$ are called  {\em matching edges}, and edges in $F_i$ are called {\em linking edges}.

Define a mapping $\sigma: V(G)-\{v^*\} \to E(G)$ which assigns to each vertex $u \in L_i$ an edge
$\sigma(u)$ of $G[L_{i-1}, L_i]$ incident to $u$: If $u $ is covered by an edge $e \in
M_i$, then let $\sigma(u)=e$, otherwise $\sigma(u)$ is an arbitrary edge of $G[L_{i-1}, L_i]$ incident to $u$.
Observe that if $u \in L_i$ is covered by a linking edge, then $u$ is also covered by a matching edge. So
$\sigma(u)$ is not a linking edge for any $u \in L_i$.

Let $G_{\sigma}[L_{i-1},L_i]$ be
the subgraph of $G[L_{i-1},L_i]$ obtained by deleting all the edges $\{\sigma(u): u \in L_i\}$ and
let $G'_{\sigma}[L_{i-1},L_i]$ be the subgraph of $G_{\sigma}[L_{i-1},L_i]$ obtained by
deleting all the linking edges.

A component  $H$ of $G'_{\sigma}[L_{i-1},L_i]$  is a {\em bad }
if $H$ is $2k$- regular, and
all vertices of $H \cap L_i$ are covered by  linking edges.

An $L_i$-link $uvu'$ in $F_i$ is {\em free}, if one of $u$
and $u'$ does not belong to any bad component of $G'_{\sigma}[L_{i-1},L_i]$.

We choose a minimum covering pair $(F_i,M_i)$ and $\sigma$ so  that  the number of free $L_i$-links is maximum.

\begin{lemma}
\label{clm23}
If $G'_{\sigma}[L_{i-1},L_i]$ has a bad component, then $F_i$ contains  at least $k $ free $L_i$-links.
\end{lemma}
\begin{proof}
Let $H$ be a bad component of $G'_{\sigma}[L_{i-1},L_i]$. For a vertex $u \in H \cap L_i$, let $u'vu$ be an $L_i$-link in $F_i$.
Then $v \notin H$ because $H$ is $2k$-regular and   $v$ has degree at most $2k-1$ in $G'_{\sigma}[L_{i-1},L_i]$.

If $v$ has a neighbour $w \in L_i$   not covered by an $L_i$-link in $F_i$, then let
$F'_i=F_i-uvu'+wvu'$.
Then $(F'_i, M_i)$ is a minimum covering pair of $G'_{\sigma}[L_{i-1},L_i]$. For this covering pair,
the component of $G'_{\sigma}[L_{i-1},L_i]$ containing $u$ is not
 bad any more.
All $L_i$-links in $F'_i$ covering vertices of $H \cap L_i$ are free links. As $H$ is $2k$-regular, $|H \cap L_i| \ge 2k$, hence there
are at least $k$ free links in $F'_i$. By our choice of $(F_i,M_i)$, we conclude that $F_i$ also has at least $k$ free links.

Assume all neighbours of $v$ are covered by $L_i$-links in $F_i$. Note that $v$ is a $(2k+1)$-vertex in $G[L_{i-1}, L_i]$ (as $(F_i. M_i)$ is a minimum covering pair).
So $v$ has $2k-1$ neighbours other than $u,u'$. All these neighbours and $v$ are in the same component of $G'_{\sigma}[L_{i-1},L_i]$, and this component is not bad (as $v$ has degree $2k-1$ in $G'_{\sigma}[L_{i-1},L_i]$). Therefore
all the $L_i$-links in $F_i$ covering these $2k-1$ neighbours of $v$ are free and $F_i$ has at least $k$ free $L_i$-links.
\end{proof}

For $i=1,2,\ldots, p$, let $n_i=|L_i|$, $a_i=|E(G[L_i])|$, $b_i = |E(G'_{\sigma}[L_{i-1},L_i])|$, and let $c_i$ be the number of {\em linking edges} in $G_{\sigma}[L_{i-1},L_i]$.
Let $d_p=0$ and for $1 \le i \le p-1$,  $d_i= d_{i+1} +  a_{i+1} +n_{i+1} + b_{i+1}+c_{i+1}$.

For integers $a < b$, $[a,b]=\{j: a \le j \le b\}$ denotes the set of  integers between $a$ and $b$.   Let
\begin{eqnarray*}
I_i &=&[d_i+1,  d_i+a_i], \\
J_i &=& [d_i+a_i+1,   d_i+a_i+b_i], \\
K_i &=& [d_i+a_i+b_i+1,   d_i+a_i+b_i+c_i],\\
L_i &=& [d_i+a_i+b_i+c_i+1,   d_i+a_i+b_i+c_i+n_{i}=d_{i-1}].
\end{eqnarray*}
Observe that $I_p,J_p,K_p,L_p,I_{p-1}, J_{p-1}, K_{p-1},L_{p-1} \ldots, I_1,J_1,K_1,L_1$ form a partition of the set
  $[1, |E|]$.
Let
\begin{eqnarray*}
E_i &=& E(G[L_i]), \\
E'_i &=& E(G'_{\sigma}[L_{i-1}, L_i]),\\
E''_i &=& E(F_i),\\
\hat{E}_i &=& \{\sigma(u): u \in L_i\}.
\end{eqnarray*}

The edge sets will be labeled in the following  order: $$E_p, E'_p, E''_p, \hat{E}, E_{p-1}, E'_{p-1}, E''_{p-1}, \hat{E}, \ldots, E_1, E'_1, E''_1,\hat{E}_1,$$
and
\begin{itemize}
\item
edges in $E_i$ will be assigned labels from $I_i$,
\item
edges in  $E'_i$ will be assigned labels from $J_i$,
\item
edges in $E''_i$ will be assigned labels from $K_i$.
\item
and edges in $\hat{E}_i$ will be assigned labels from $L_i$.
\end{itemize}

All labelings in our discussion are assumed to use the desribed label sets for corresponding edges.

The assignment of labels from $I_i$ to edges in $E_i$ is arbitrary.
However, the assignment of labels from $J_i$ to edges in $E'_i$ and the labels from $K_i$ to edges in $E''_i$ are carefully constructed.
The labeling of edges in $\hat{E}$ is not arbitrary but easily defined.

We denote by $f$ the final edge labeling. Recall that $\varphi_f(u) = \sum_{e \in E(u)}f(e)$ is the
sum the labels of the $k$ edges incident to $u$. For $u \ne v^*$, let $$s(u) = \sum_{e \in E(u)
\setminus \{\sigma(u)\}} f(e) = \varphi_f(u) - f(\sigma(u)).$$ Observe that after assigning labels
to  the edges in $E''_i$ (according to the above order), $s(u)$ is determined for each vertex $u
\in L_i$.

\begin{lemma}
\label{su}
The edges in $E'_i \cup E''_i$ can be labeled  in such a way that the following hold:
If  edges in $E'_{i+1},   E''_{i+1}, E_i, \hat{E}_i, E_{i-1}, E'_{i-1}, E''_{i-1}$ are labeled by
labels from $J_{i+1}, K_{i+1}, I_i, L_i, I_{i-1}, J_{i-1}, K_{i-1}$, respectively,
then for any $u \in L_i$,   $$s(u) \le
(2k+1)(d_i+a_i)+(k+1)b_i+c_i+k \eqno(4)$$
and for any $v\in L_{i-1}$, $$s(v) \ge  (2k+1)(d_i+a_i)+(k+1)b_i+c_i+k. \eqno(5)$$
\end{lemma}

Once Lemma \ref{su} is proved, we extend  $f$  to edges in $\hat{E}_i$ as follows:
We order the vertices in $L_i$ as $v_1, v_2, \ldots, v_q$ in such that that $s(v_j) \le s(v_{j+1})$.
Label the edges  $\hat{E}_i$ in such that way that $f(\sigma(v_{j+1}) =f(\sigma(v_j))+1$. Then  $\varphi_{f}(u) \ne \varphi_{f}(u')$ for
distinct
 vertices   $u, u' \in L_i$ and hence $f$ is an antimagic labeling of $G$. For $u \in L_i$ and $v \in L_{i-1}$,
since $f(\sigma(u)) < f(\sigma(v))$, we have
$\varphi_f(u) < \varphi_f(v)$.

\bigskip
So all it remains is to prove Lemma \ref{su}.
\bigskip

As mentioned above, edges in $E_i$ are assigned labels from $I_i$ in an arbitrary manner.
We describe how the edges in $E'_i\cup E''_i$ are labeled.

We decompose each component of $G'_{\sigma}[L_{i-1}, L_i]$ into the edge disjoint union of open trails and closed trails:
If a component $H$   of $G'_{\sigma}[L_{i-1}, L_i]$ is Eulerian, then $E(H)$ form a closed trail.
Otherwise, $H$ is the edge disjoint union of a family of open trails with end vertices having odd degree in $H$
and  each odd degree  vertex of $H$ is the end of exactly one open trail.

Let ${\cal C}_i$ be the set of closed trails. We divide the open  trails into three families:
\begin{itemize}
\item ${\cal P}_{i,1}$ are open trails in $G'_{\sigma}[L_{i-1}, L_i]$   with both end vertices in $L_{i-1}$.
\item ${\cal P}_{i,2}$ are open trails in $G'_{\sigma}[L_{i-1}, L_i]$  with both end vertices in $L_{i}$.
\item ${\cal P}_{i,3}$ are open trails in $G'_{\sigma}[L_{i-1}, L_i]$  with one end vertex in $L_{i-1}$ and one end
vertex in $L_i$.
\end{itemize}

We label the edges of the closed trails in ${\cal C}_i$ and open trails in ${\cal P}_{i,1} \cup {\cal P}_{i,2}$ one by one, and label the trails
in ${\cal P}_{i,3}$ one pair each time, except that in case $|{\cal P}_{i,3}|$ is odd, then the last trail is labeled by itself.
The label uses a set of  {\em permissible labels}.
Once a label is used,
then it is removed from the set of permissible labels.
The set of permissible labels for each closed trail (or each trail in  ${\cal P}_{i,1} \cup {\cal P}_{i,2}$, or each pair of trails in ${\cal P}_{i,3}$,
or the last trail in ${\cal P}_{i,3}$, in case $|{\cal P}_{i,3}|$ is odd),
is always an interval $[s,l]$ of integers with
$s+l= 2d_i+2a_i+b_i+1$. For simplicity, we write the interval always as $[s,l]$, however, $s$ and $l$ changes for different trails and cycles, but
the sum $s+l=2(d_i+a_i)+b_i+1$ remains unchanged.

\begin{enumerate}
\item  For an open trail
$P=(u_1,u_2,\ldots,u_{2q+1})$ in ${\cal P}_{i,1}$, with $u_1 \in L_{i-1}$,
  we label the edges of $P$ with labels $s,l,s+1,l-1,s+2, \ldots, s+q-1, l-q+1$ in this order.
\item For an open trail
$P=(u_1,u_2,\ldots,u_{2q+1})$ in ${\cal P}_{i,2}$, with $u_1 \in L_{i}$,
we label the edges of $P$ with labels $l,s,l-1,s+1,l-2,s+2, \ldots, l-q+1, s+q-1$ in this order.
\item  For a pair of open trails
$P=(u_1,u_2,\ldots,u_{2q})$ and $P'=(u'_1, u'_2, \ldots, u'_{2q'})$ in ${\cal P}_{i,3}$,
we choose the initial vertex of $P$ and $P'$ in such a way that    $u_{2j}, u'_{2j+1} \in L_{i}$,
and  label the edges of $P \cup P'$ with labels $s,l,s+1,l-1,s+2, \ldots, l-(q+q'-2), s+(q+q'-2)$, in this order (with edges of $P$ in front of edges of $P'$).
If $|{\cal P}_{i,3}|$ is odd, then
  the last trail in ${\cal P}_{i,3}$ is not paired off with another trail in ${\cal P}_{i,3}$.
Assume the trail is $P=(u_1,u_2,\ldots,u_{2q})$.   We label the edges of $P$, using labels
$s,l,s+1,l-1,s+2, \ldots, l-(q-1), s+(q-1)$ in this order.
\item Assume $W=(u_0,u_1,u_2,\ldots,u_{2q}=u_0)$ is a closed trail, which covers
an Eulerian component $H$ of   $G'_{\sigma}[L_{i-1}, L_i]$.
\begin{itemize}
\item[(i)] If $H$ is not bad, then  we can choose $u_0$ such that one of the following holds:
\begin{enumerate}
\item[(A)]  $u_0 \in L_i$ and either  ${\rm deg}_{H}(u_0) \le 2k-1$ or $u_0$ is not covered by an $L_i$-link.
\item[(B)] $u_0 \in L_{i-1}$, ${\rm deg}_{H}(u_0) \le 2k-1$.
\end{enumerate}
If (A) holds, then label
 the edges of $W$ with labels $l, s, l-1, s+1, l-2, s+2 \ldots,  l-q+1, s+q-1,$ in this order, i.e.,
the first edge receives label $l$, the second receives label $s$, etc.

If (B) holds, then
label the edges of $C'$ with labels $s,l,s+1,l-1,s+2, \ldots, s+q-1, l-q+1$ in this order, i.e.,
the first edge receives label $s$, the second receives label $l$, etc.
\item[(ii)] If $H$ is a bad component, then assume  $u_0 \in L_i$, and we label the edges with labels $s, l, s+1, l-1, \ldots, s+q-1, l-q+1$ in this order.
\end{itemize}
\item To label edges in $F_i$, we first order links in $F_i$ as $F_i=\{u_jv_ju'_j: j=1,2, \ldots, |F_i|\}$ in such a way that  free $L_i$-links (if any)   are listed at the beginning.
If $u_jv_ju'_j$ is a free $L_i$-link, then we assume that $u'_j$ is not contained in a bad component.
Label the  edges   in $F_i$ as
$$f(u_jv_j) = d_i+a_i+b_i+j, \\ f(u'_jv_j) = d_i+a_i+b_i+c_i -j+1.$$
\end{enumerate}

Figures below show the labeling of the trails.

\begin{figure}[ht]
\begin{center}
\scalebox{0.36}{\includegraphics{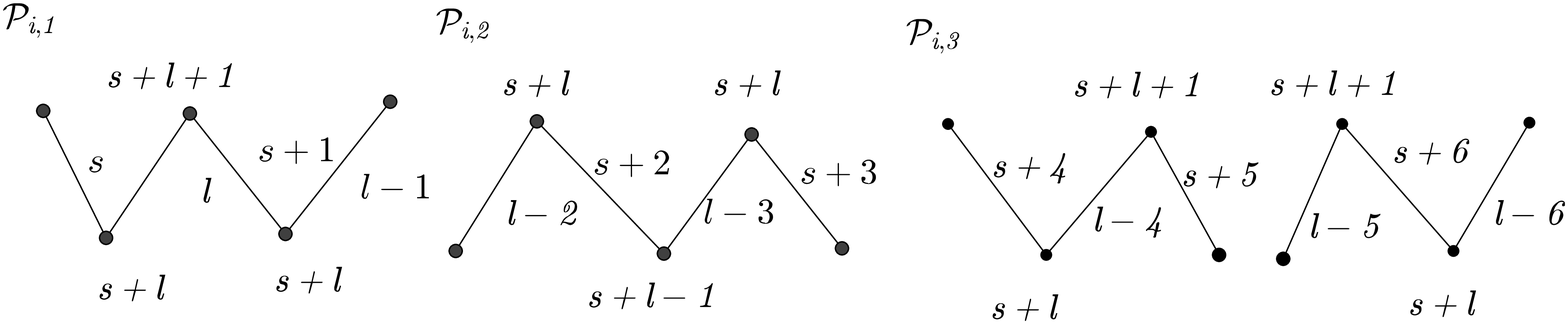}}
\caption{Open trails in ${\cal P}_{i,j}$.}
\end{center}
\end{figure}

\begin{figure}[ht]
\begin{center}
\scalebox{0.7}{\includegraphics{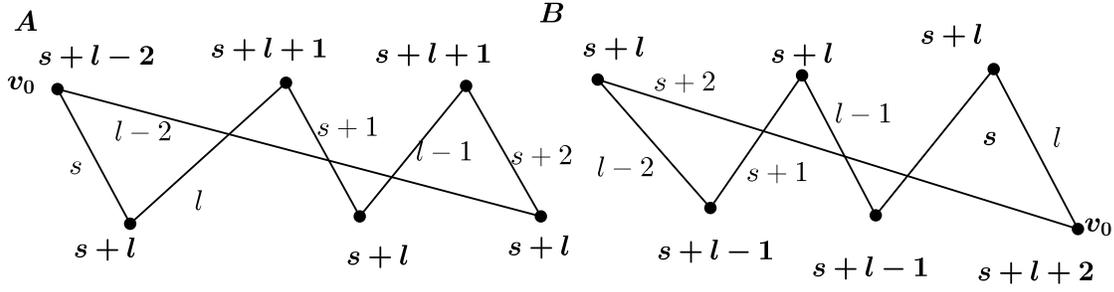}}
\caption{Closed trails as in 4(i)(A) and 4(i)(B).}
\end{center}
\end{figure}

\begin{lemma}
\label{bad}
If $u_jv_j$ is a link edge
and $u_j$ is contained in a bad component, then $f(u_jv_j) \le d_i+a_i+b_i+c_i -k$.
\end{lemma}
\begin{proof}
By Lemma \ref{clm23}, if $G'_{\sigma}[L_{i-1},L_i]$ has a bad component, then $F_i$ contains  at least $k $ free $L_i$-links.
So  for $j=1,2,\ldots, k$, $u_jv_ju'_j$ are free $L_i$-links.
If $j \le k$, then $f(u_jv_j) =d_i+a_i+b_i+j \le d_i+a_i+b_i+c_i -k$ (note that $c_i = 2|F_i|$).
If $j > k$, then $f(u_jv_j) =d_i+a_i+b_i+c_i-j +1 \le d_i+a_i+b_i+c_i -k$.
\end{proof}

\begin{figure}[ht]
\begin{center}
\scalebox{0.6}{\includegraphics{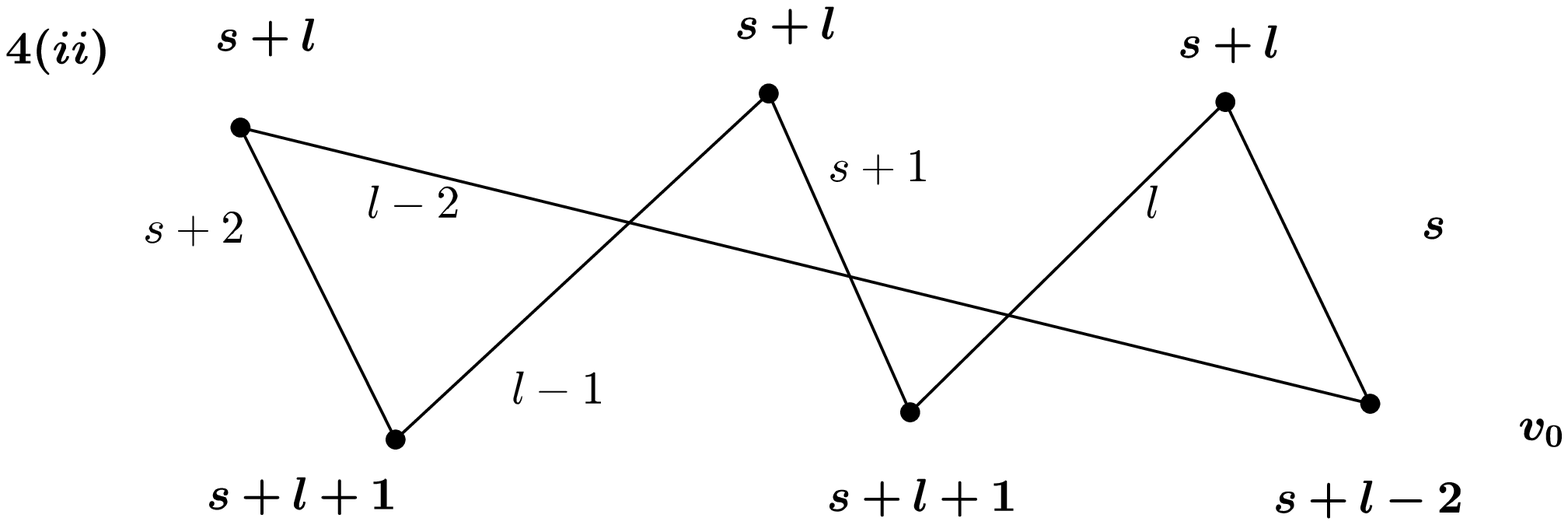}}
\caption{A closed trail as in 4(ii).}
\end{center}
\end{figure}

\bigskip
{\noindent}
{\bf Proof of Lemma \ref{su}}: For each vertex $v$, let $E(v)$ be the edges of $G$ incident to $v$, and let $S(v)=E(v) - \{\sigma(v)\}$.
By definition, $$s(v) = \sum_{e \in S(v)}f(e).$$

First we prove inequality (5).

Assume $v \in L_{i-1}$.  Let
\begin{eqnarray*}
S_1(v) &=& \{e \in S(v) \cap W: W \in {\cal C}_i \cup {\cal P}_{i,1} \cup {\cal P}_{i,2} \cup {\cal P}_{i,3}\},\\
S_2(v) &=& S(v) \cap \left(E_{i-1} \cup E'_{i-1} \cup E''_{i-1} \cup \hat{E}_i\right),\\
S_3(v) &=& S(v) \cap  E''_i.
\end{eqnarray*}
Then $S(v) = S_1(v) \cup S_2(v) \cup S_3(v)$. Note that   $S_3(v)$ might be empty.

We pair off the edges in $S(v)$   as follows:
\begin{itemize}
\item  Two consecutive edges in a trail in $S_1(v)$ form a pair.
\item  If $v$ is the end vertex of a trail $W$, then the end edge of $W$ in $S_1(v)$ is paired with an arbitrary edge in $S_2(v)$ (if $S_2(v)  \ne \emptyset$),
\item The remaining edges of $S_2(v)$ (if any) are paired off arbitrarily.
\item Edges in $S_3(v)$ form a pair (if it is not empty).
\end{itemize}

Note that $|S(v)|=2k+1$ is odd. So exactly one edge is left over. This left over edge is an edge in $S_2(v)$ if $S_2(v) \ne \emptyset$ and is
an end edge of a trail otherwise.

It is straightforward to verify that each pair of edges have label sum at least $2d_i+2a_i+b_i+1$
(refer to Figures 1,2,3 for pairs contained in trails), with one exception:

If $v=u_0 \in W$ as in 4(i)(B), then the two consecutive edges of $W$
incident to $v$ (i.e., the first and the last edges of $W$) have label sum   $s+l-q+1 \ge  2d_i+2a_i +1$.

Assume first it is not the exception case.
If the left over edge is in $S_2(v)$, then its label is at least
$d_i+a_i+b_i+c_i+1$. So $$s(v) \ge k(2d_i+2a_i+b_i+1)+d_i+a_i+b_i+c_i+1,$$
hence inequality (5) holds.

If the left over edge is in $S_1(v)$, then $S_3(v) \ne \emptyset$ (because $v$ is covered by $F_i \cup M_i$). The pair in $S_3(v)$ have label sum $2d_i+2a_i+2b_i+c_i+1$.
Hence the left over edge together with edges in $S_3(v)$ have label sum at least $3d_i+3a_i+2b_i+c_i+2$. So
$$s(v) \ge (k-1)(2d_i+2a_i+b_i+1)+3d_i+3a_i+2b_i+c_i+2,$$
again inequality (5) holds.

In the exceptional case, since $d_H(v) \le 2k-1$,  $|S_3(v) \cup S_2(v)| \ge 2$. Hence there is a pair whose label sum is at least $2d_i+2a_i+2b_i+c_i+1$. This compensate the exceptional pair,
and hence inequality (5) still holds.


Next we prove inequality (4).

Assume $v \in L_i$. Let
\begin{eqnarray*}
S_1(v) &=& \{e \in S(v) \cap W: W \in {\cal C}_i \cup {\cal P}_{i,1} \cup {\cal P}_{i,2} \cup {\cal P}_{i,3}\},\\
S_2(v) &=& S(v) \cap \left(E_{i} \cup E'_{i+1} \cup E''_{i+1} \cup \hat{E}_{i+1}\right),\\
S_3(v) &=& S(v) \cap  E''_i.
\end{eqnarray*}
Then $S(v) = S_1(v) \cup S_2(v) \cup S_3(v)$.
Note that   $|S_3(v)| \le 1$.

Then we pair off the edges in $S(v)$   as follows:
\begin{itemize}
\item  Two consecutive edges in a trail in $S_1(v)$ form a pair.
\item  If $v$ is the end vertex of a trail $W$, then the end edge of $W$ in $S_1(v)$ is paired with an arbitrary edge in $S_2(v)$ (if $S_2(v)  \ne \emptyset$),
\item The remaining edges of $S_2(v)$ (if any) are paired off arbitrarily.
\end{itemize}

Again, exactly one edge is   left over. If $S_3(v) \ne \emptyset$, then the left over edge is the edge in $S_3(v)$.

It is straightforward to verify that each pair of edges have label sum at most $2d_i+2a_i+b_i+1$,
with two exception:
\begin{itemize}
\item   If $v=u_0$ in the closed trail $W$ as in 4(i)(A), then the two consecutive edges of $W$
incident to $v$ (i.e., the first and the last edges of $W$) have label sum   $s+l+q-1 \le  2d_i+2a_i+2b_i $.
\item If $v \in W$ and $W$ form a bad component, then   each pair of edges in $S_1(v)$ have label sum $2d_i+2a_i+b_i+2$.
\end{itemize}

Assume it is not an exceptional case. As the left over edge has label at most $d_i+a_i+b_i+c_i$, $$s(v) \le k( 2d_i+2a_i+b_i +1)+d_i+a_i+b_i+c_i,$$
hence inequality (4) holds.

In the first exceptional case, $S_2(v)$ contains a pair whose label sum is at most $2d_i+2a_i$.
 This compensate the exceptional pair,
and hence inequality (4) still holds.

In the second exceptional case, by definition $S_3(v) \ne \emptyset$ and by Lemma \ref{bad}, the label of the edge in $S_3(v)$ is at most $d_i+a_i+b_i+c_i-k$.
So $$s(v) \le k( 2d_i+2a_i+2b_i +2)+d_i+a_i+b_i+c_i-k,$$
and inequality (4) holds.

This completes the proof of Lemma \ref{su}, and hence the proof of Theorem \ref{even}.


\begin{thebibliography}{99}

\bibitem{AKLRY2004} N. Alon, G. Kaplan, A. Lev, Y. Roditty, and R.
Yuster, {\em Dense graphs are antimagic}, J. Graph Theory,
47 (2004), 297-309.


\bibitem{Cra2009} D. W. Cranston, {\em Regular bipartite graphs are antimagic},
J. Graph Theory 60 (2009), 173-182.

\bibitem{CLZ} D. W. Cranston, Y. Liang and X. Zhu, {\em Odd degree regular bipartite graphs are antimagic},
 J Graph Theory, to appear.

\bibitem{Ecc2014} T. Eccles, {\em Graphs of large linear size are antimagic}, manuscript, 2014,
arXiv:1409.3659



\bibitem{survey} J. A. Gallian, {\em A Dynamic survey on Graph Labeling}, The Electronic
Journal of Combinatorics, 15 (2008).



\bibitem{HR1990} N. Hartsfield and G. Ringel, {\em Pearls in Graph Theory}, Academic
Press, INC., Boston, 1990, pp. 108-109, Revised version 1994.

\bibitem{Dan2005} D. Hefetz, {\em Anti-magic graphs via the Combinatorial Nullstellensatz}, J Graph Theory 50 (2005), 263-272



\bibitem{KLR2009} G. Kaplan, A. Lev and Y. Roditty, {\em On zero-sum
partitions and antimagic trees}, Discrete Math. 309 (2009), 2010-2014.

\bibitem{konig1916} D. K\"{o}nig,
\"{U}ber Graphen und ihre Anwendung auf Determinantentheorie und Mengenlehre.
Math. Ann. 77 (1916),  453-465.

\bibitem{konig1936} D. K\"{o}nig,
Theorie der endlichen und unendlichen Graphen. Kombinatorische Topologie der Streckenkomplexe. Chelsea Publishing Co., New York, N. Y., 1950.


\bibitem{LWZ2012} Y. Liang, T. Wong and X. Zhu, {\em Anti-magic labeling of trees}, Discrete Math. 331 (2014), 9-14.



\bibitem{LZ2014} Y. Liang and X. Zhu, {\em Anti-magic labeling of cubic graphs}, J. Graph Theory, 75 (2014), 31-36.

\bibitem{LZ2013} Y. Liang and X. Zhu, {\em Anti-magic labelling of Cartesian product of graphs}, Theoret. Comput. Sci.
 477 (2013), 1-5.

\bibitem{Rin1964} G. Ringel, Problem 25, in Theory of Graphs and its Applications, Proc. Symposium
Smolenice 1963, Prague (1964), 162.




\end{thebibliography}
\end{document}